\documentclass[reqno]{amsart} 	

\usepackage{amsmath,amssymb,amsfonts,amsthm}
\usepackage{a4wide}

\usepackage{subcaption}
\usepackage{graphicx}

\theoremstyle{plain}
\newtheorem{theorem}{Theorem}
\newtheorem{lemma}[theorem]{Lemma}
\newtheorem{ex}[theorem]{Example}
\theoremstyle{remark}

\theoremstyle{definition}
\newtheorem{definition}{Definition}

\begin{document}

\title[Permanence and uniform asymptotic stability of positive solutions]{%
Permanence and uniform asymptotic stability of positive solutions of SAIQH models on time scales}

\author[N. Zine]{Nedjoua Zine}
\address[N. Zine]{University Mustapha Stambouli of Mascara, Algeria}
\email{nadjoua.zine@univ-mascara.dz}

\author[B. Bayour]{Benaoumeur Bayour}
\address[B. Bayour]{University Mustapha Stambouli of Mascara, Algeria}
\email{b.bayour@univ-mascara.dz}

\author[D. F. M. Torres]{Delfim F. M. Torres}
\address[D. F. M. Torres]{%
Center for Research and Development in Mathematics and Applications (CIDMA),
Department of Mathematics, University of Aveiro, 3810-193 Aveiro, Portugal}
\email{delfim@ua.pt}

\address[D. F. M. Torres]{Research Center in Exact Sciences (CICE),
Faculty of Sciences and Technology (FCT),
University of Cape Verde (Uni-CV),
7943-010 Praia, Cape Verde}
\email{delfim@unicv.cv}


\begin{abstract}
A Susceptible--Asymptomatic--Infectious--Quarantined--Hospitalized (SAIQH)
compartmental model on time scales is introduced and
a suitable Lyapunov function is defined. Main results include:
the proof that the system is permanent; proof of existence of solution;
and sufficient conditions implying the dynamic system to have a unique
almost periodic solution that is uniformly asymptotically stable. 
An example is presented supporting the obtained results.
\end{abstract}



\date{Submitted January 31, 2024}

\maketitle

\numberwithin{equation}{section}
\allowdisplaybreaks


\section{Introduction}

In 2019, the COVID-19 pandemic has appeared for the first time in Wuhan, China,
attracting many researchers to investigate outbreaks and the spread
of viruses \cite{MyID:461,MyID:468}. Some of the studies provided
new mathematical compartmental models, illustrating well the important
contributions of Mathematics to fight communicable diseases.
Such models have been used to analyze the corresponding dynamics and to supply
useful techniques in disease epidemiology \cite{[b7],[b5],[b6]}.

Kim et al. studied a SAIQH (Susceptible--Asymptomatic--Infectious--Quarantined--Hospitalized)
mathematical model to analyze the transmission dynamics of MERS and to estimate
transmission rates \cite{[kim]}. Lemos-Pai\~{a}o et al. developed a SAIQH type model for COVID-19,
which was important to describe and understand the pandemic in Portugal \cite{[b5]}.
Similar to \cite{[b5]}, here we also consider the $H_{IC}$ class of hospitalized individuals
in intensive care units.

Let the total living population under study, denoted by $N(t)$, $t\geq 0$,
be divided into six classes: (i)~the susceptible individuals $S(t)$;
(ii)~the infected individuals without (or with mild) symptoms $A(t)$ (the  Asymptomatic);
(iii)~infected individuals $I(t)$ with visible symptoms; (iv)~quarantined individuals
$Q(t)$ in isolation at home; (v)~hospitalized individuals $H(t)$;
(vi)~and hospitalized individuals $H_{IC}(t)$ in intensive care units.
The mathematical model introduced and studied in \cite{[b5]} reads:
\begin{equation}
\label{1}
\begin{cases}
\dot{S}(t)=\Lambda+\omega nQ(t)-[\lambda(t)(1-p)+\phi p+\gamma]S(t), \\
\dot{A}(t)=\lambda(t)(1-p)S(t)-[q\nu+\gamma]A(t),\\
\dot{I}(t)=q\nu A(t)-[\delta_{1}+\gamma]I(t),\\
\dot{Q}(t)=\phi pS(t)+\delta_{1}f_{1}I(t)
+\delta_{2}(1-f_{2}-f_{3})H(t)-[\omega m+\gamma]Q(t),\\
\dot{H}(t)=\delta_{1}(1-f_{1})I(t)+\eta(1-k)H_{IC}(t)
-\left[\delta_{2}(1-f_{2}-f_{3})+\delta_{2}f_{2}+\alpha_{1}f_{3}+\gamma\right]H(t),\\
\dot{H_{IC}}(t)=\delta_{2}f_{2}H(t)-[\eta(1-k)+\alpha_{2}k+\gamma]H_{IC}(t),
\end{cases}
\end{equation}
where, for all time $t\geq 0$,
$$
\lambda(t)=\dfrac{\beta\left(l_{A}A(t)+I(t)+l_{H}H(t)\right)}{N(t)}
$$
is a bounded positive function with
$$
N(t)=S(t)+A(t)+I(t)+Q(t)+H(t)+H_{IC}(t),
$$
$\beta,l_A,l_H > 0$, and all the other parameters in the model are nonnegative.
In addition,
$$
p,1-p,k,1-k,q,f_1,1-f_1,f_2,f_3,1-f_2-f_3\in[0,1].
$$
For more details on the mathematical model \eqref{1}
we refer the reader to \cite{[b5]}.

Nowadays, the study of dynamic equations on time scales is a leading topic
of research \cite{[b1],[b0],[Martin],[Martin2],[khudd],[khudd1],[khudd2],[Tisdell],[b18]}.
Motivated by the above mentioned references and the 2019 paper of Khuddush and Prasad
on a $n$-species Lotka--Volterra system \cite{[khudd]}, here we extend  \eqref{1}
to a time scale $\mathbb{T}$, studying the permanence and uniform asymptotic stability
of the unique positive almost periodic solution of the following SAIQH
type model on time scales:
\begin{equation}
\label{sy}
\begin{cases}
S^{\Delta}(t)=\Lambda+\omega nQ(t)-[\lambda(t)(1-p)+\phi p+\gamma]S^{\sigma}(t),\\
A^{\Delta}(t)=\lambda(t)(1-p)S(t)-[q\nu+\gamma]A^{\sigma}(t),\\
I^{\Delta}(t)=q\nu A(t)-[\delta_{1}+\gamma]I^{\sigma}(t),\\
Q^{\Delta}(t)=\phi pS(t)+\delta_{1}f_{1}I(t)+\delta_{2}(1-f_{2}-f_{3})H(t)
-[\omega m+\gamma]Q^{\sigma}(t),\\
H^{\Delta}(t)=\delta_{1}(1-f_{1})I(t)+\eta(1-k)H_{IC}(t)
-\left[\delta_{2}(1-f_{2}-f_{3})+\delta_{2}f_{2}+\alpha_{1}f_{3}+\gamma\right]H^{\sigma}(t),\\
H_{IC}^{\Delta}(t)=\delta_{2}f_{2}H(t)-[\eta(1-k)+\alpha_{2}k+\gamma]H_{IC}^{\sigma}(t),
\end{cases}
\end{equation}
where all parameters are real nonnegative numbers
and $\lambda(t)$ is bounded positive,
as mentioned before. In the particular case $\mathbb{T} = \mathbb{R}^+$,
system \eqref{sy} reduces to \eqref{1}.

The text is organized as follows.
In Section~\ref{sec:2}, we recall the main
necessary notions and results needed in the sequel.
The original results are then formulated and proved in Section~\ref{sec:3}:
we show that system (\ref{sy}) is permanent (Theorem~\ref{2});
we give a sufficient condition for existence of a solution to system (\ref{sy})
(Theorem~\ref{thm:existence}); and we obtain conditions implying the dynamic system (\ref{sy})
to have a unique almost periodic solution that is uniformly asymptotically stable	
(Theorem~\ref{impth}).

For convenience, in the sequel we put
$x_{1} = S$,
$x_{2} = A$,
$x_{3} = I$,
$x_{4} = Q$,
$x_{5} = H$,
and $x_{6} = H_{IC}$.


\section{Preliminaries}
\label{sec:2}

For a function $f(t)$ defined on $t \in \mathbb{T}^{+}$, we set
$$
f^{L} := \inf \left\lbrace f(t): t \in \mathbb{T}^{+}\right\rbrace,
\; \; f^{U}:=\sup \left\lbrace f(t): \, t \in \mathbb{T}^{+}\right\rbrace.
$$

We begin by recalling some basic but fundamental concepts of the calculus on time scales \cite{[b1]}.
A time scale $\mathbb{T}$ is a nonempty closed subset of the real numbers $\mathbb{R}$;
on $\mathbb{T} $ one considers the topology inherited from the real numbers
with the standard topology. The jump operators
$ \sigma, \rho: \, \mathbb{T} \longrightarrow \mathbb{T}$ and the graininess
function $\mu: \mathbb{T}\longrightarrow \mathbb{R}^{+}$ are defined, respectively, by
$$
\sigma(t)=\inf\left\lbrace \tau \in \mathbb{T}: \tau >t\right\rbrace , \;
\rho(t)=\sup \left\lbrace \tau \in \mathbb{T}: \tau < t\right\rbrace, \;
\text{ and } \; \mu(t)= \sigma(t)-t,
$$
supplemented by  $\inf \emptyset = \sup \mathbb{T}$ and $\sup \emptyset = \inf \mathbb{T}$.
A point $t \in \mathbb{T} $ is left-dense, left-scattered, right-dense, or right-scattered
when $\rho(t)=t$, $\rho(t)<t$, $\sigma(t)=t$, and $\sigma(t)>t$, respectively.
If $\mathbb{T} $ has a left-scattered maximum $ m$, then
$\mathbb{T}^{\kappa}:=\mathbb{T}\setminus \left\lbrace m \right\rbrace$;
otherwise, $\mathbb{T}^{\kappa}:=\mathbb{T}$. If $f:\mathbb{T}\rightarrow \mathbb{R}$,
then $f^{\sigma}:\mathbb{T}\rightarrow \mathbb{R}$
is given by $f^{\sigma}(t)=f(\sigma(t))$ for all $t\in\mathbb{T}$.	

\begin{definition}
A function $f: \mathbb{T}\longrightarrow \mathbb{R}$
is called regressive provided
$1+ \mu(t)f(t)\neq 0 $ for all $t \in \mathbb{T}^{\kappa}$;
a function $g: \mathbb{T}\longrightarrow \mathbb{R}$ is called rd-continuous provided
it is continuous at right-dense points in $ \mathbb{T}$ and its left-sided
limits exist (finite) at left-dense points in $ \mathbb{T}$.
The set of all regressive and rd-continuous functions
$f: \mathbb{T}\longrightarrow \mathbb{R}$
will be denoted by $ \mathcal{R}=\mathcal{R}(\mathbb{T},\mathbb{R})$.
We define the set $\mathcal{R}^{+}= \mathcal{R}^{+}(\mathbb{T},\mathbb{R})
=\left\lbrace f \in \mathcal{R}: 1+\mu(t)f(t)>0, \; \text{for} \, \text{all}
\; t \in \mathbb{T}\right\rbrace$.	
\end{definition}

\begin{lemma}[See \cite{[b4]}]
\label{23}
Assume that $\alpha >0$, $b>0$,
and $-\alpha \in \mathcal{R}^{+}$.
If
$$
y^{\Delta}(t)\geq (\leq) \,
b-\alpha y^{\sigma}(t),
\quad y(t)>0,
\quad t \in [t_{0},\infty)_{\mathbb{T}},
$$
then
$$
y(t) \geq (\leq) \,
\dfrac{b}{\alpha} \left[ 1+\left( \dfrac{\alpha y(t_{0})}{b}
-1\right)e_{(-\alpha)}(t,t_{0})\right],
\quad t \in [t_{0}, \infty)_{\mathbb{T}},
$$
where $e_{\cdot}(\cdot,\cdot)$ is the standard exponential function
of the time-scale calculus \cite{[b1]}.
\end{lemma}

\begin{definition}[See \cite{[b9]}]
A time scale $ \mathbb{T}$ is called an almost periodic time scale if
$$
\Pi= \left\lbrace \tau \in \mathbb{R}: \, t+\tau \in \mathbb{T}, \,
\text{for} \, \text{all} \, \, t \in \mathbb{T}\right\rbrace
\neq \left\lbrace 0 \right\rbrace.
$$
\end{definition}

\begin{definition}[See \cite{[b9]}]
Let $ \mathbb{T} $ be an almost periodic time scale.
A function $ x \in C( \mathbb{T},\mathbb{R}^{n})$
is called an almost periodic function if the
$\varepsilon$-translation set of $x$, that is,
$$
E\left\lbrace \varepsilon, x\right\rbrace =\left\lbrace \tau \in \Pi
: \, \left| x(t+\tau)-x(t)\right| <\varepsilon \, \text{ for} \, \text{all}
\, \,  t \in \mathbb{T} \right\rbrace,
$$
is a relatively dense set in $\mathbb{T}$
for all $\varepsilon >0$ and there exists a constant $l(\varepsilon)>0$
such that each interval of length $l(\varepsilon)$ contains a
$\tau \in E \left\lbrace \varepsilon,x\right\rbrace$
for which $ \lvert x(t+\tau)-x(t)\lvert < \varepsilon$
for all $ t \in \mathbb{T}$.
The value $\tau $ is known as the $\varepsilon$-translation number of $x$
and $l(\varepsilon)$ is called the inclusion length of
$E\left\lbrace \varepsilon,x\right\rbrace$.
\end{definition}

\begin{definition}[See \cite{[b9]}]
Let $\mathbb{D}$ be an open set in $\mathbb{R}^{n}$
and let $\mathbb{T}$ be a positive almost periodic time scale.
A function $f \in C(\mathbb{T}\times \mathbb{D},\mathbb{R}^{n})$
is called an almost periodic function in $t \in \mathbb{T}$
uniformly for $x \in \mathbb{D}$ if the $ \varepsilon$-translation set of $f$,
$$
E \left\lbrace \varepsilon, f, \mathbb{S}\right\rbrace
= \left\lbrace \tau \in \Pi: \mid f(t+\tau)-f(t)\mid <\varepsilon, \,
\text{ for } \, \text{ all } \, \, (t,x)\in \mathbb{T}\times \mathbb{S}
\right\rbrace,
$$
is a relatively dense set in $\mathbb{T}$ for all $\varepsilon>0$
and, for each compact subset $\mathbb{S}$ of $ \mathbb{D}$,
that is, for any given $ \varepsilon>0 $ and each compact subset
$\mathbb{S}$ of $\mathbb{D}$, there exists a constant
$l(\varepsilon,\mathbb{S})>0$ such that each interval
of length $l(\varepsilon,\mathbb{S})$
contains $\tau(\varepsilon,\mathbb{S})
\in E\left\lbrace \varepsilon, f,\mathbb{S}\right\rbrace$
for which
$$
\mid f(t+\tau,x)-f(t,x)\mid < \varepsilon, \, \,
\text{ for } \, \text{ all } \, \,
(t,x)\in \mathbb{T} \times \mathbb{S}.
$$
\end{definition}

Consider a system
\begin{equation}
\label{eq:h}
x^{\Delta}(t)=h(t,x),
\end{equation}
where
$h: \mathbb{T}^{+} \times \mathbb{S}_{B}\longrightarrow \mathbb{R}^{n},
\; \mathbb{S}_{B}= \left\lbrace
x \in \mathbb{R}^{n} : \, \| x \| < B \right\rbrace$
and $h(t,x)$ is almost periodic in $t$ uniformly for $x \in \mathbb{S}_{B}$
and is continuous in $x$. In \cite{[b18]} the question of existence of a
unique almost periodic solution $ f(t) \in \mathbb{S}$ of (\ref{eq:h}),
which is uniformly asymptotically stable, is investigated. For our
model, we obtain from \cite{[b18]} the following result.

\begin{lemma}[cf. \cite{[b18]}]
\label{24}
Suppose that there exists a Lyapunov function $V(t,x,z) $ defined on
$\mathbb{T}^{+} \times \mathbb{S}_{B}\times \mathbb{S}_{B}$
satisfying the following conditions:
\begin{eqnarray*}
&\text{(i)}& a(\| x-z\|) \leq V(t,x,z)\leq b(\| x-z \|), \;
\text{ where } \; a,b \in \mathbb{K}\\
& &\text{ with } \mathbb{K}=\left\lbrace
\alpha \in C(\mathbb{R}^{+}, \mathbb{R}^{+}): \alpha(0)=0
\text{ and } \; \alpha \; \text{ is increasing} \right\rbrace; \\
&\text{(ii)}& | V(t,x,z)-V(t,x_{1},z_{1}) |
\leq L (\| x-x_{1}\| +\| z-z_{1}\|), \;
\text{ where } \; L >0 \, \text{ is a constant};\\
&\text{(iii)}& D^{+}V^{\Delta}(t,x,z)\leq -c V(t,x,z), \;
\text{ where } \; c>0, \ -c \in \mathcal{R^{+}},
\end{eqnarray*}
and $D^{+}V^{\Delta}$ is the Dini derivative of $V$.
Furthermore, if there exists a solution $x(t) \in \mathbb{S}$ of system
\begin{equation}
\label{1.1}
\begin{cases}
x^{\Delta}_{1}(t)=\Lambda+\omega nx_{4}(t)-[\lambda(t)(1-p)+\phi p+\gamma]x_{1}^{\sigma}(t),\\
x^{\Delta}_{2}(t)=\lambda(t)(1-p)x_{1}(t)-[q\nu+\gamma]x_{2}^{\sigma}(t),\\
x^{\Delta}_{3}(t)=q\nu x_{2}(t)-[\delta_{1}+\gamma]x_{3}^{\sigma}(t),\\
x^{\Delta}_{4}(t)=\phi px_{1}(t)+\delta_{1}f_{1}x_{3}(t)
+\delta_{2}(1-f_{2}-f_{3})x_{5}(t)-[\omega n+\gamma]x_{4}^{\sigma}(t),\\
x^{\Delta}_{5}(t)=\delta_{1}(1-f_{1})x_{3}(t)+\eta(1-k)x_{6}(t)
-[\delta_{2}(1-f_{2}-f_{3})+\delta_{2}f_{2}+\alpha_{1}f_{3}+\gamma]x_{5}^{\sigma}(t),\\
x^{\Delta}_{6}(t)=\delta_{2}f_{2}x_{5}(t)-[\eta(1-k)+\alpha_{2}k+\gamma]x_{6}^{\sigma}(t),
\end{cases}
\end{equation}
for $t\in \mathbb{T^{+}}$, where $ \mathbb{S}\cup \mathbb{S}_{B}$ is a compact set,
then there exists a unique almost periodic solution $ f(t) \in \mathbb{S}$ of system (\ref{1.1}),
which is uniformly asymptotically stable.
\end{lemma}


\section{Main Results}
\label{sec:3}

Let $t_{0} \in \mathbb{T}$ be a fixed positive initial time.
Our main results are: the proof that system (\ref{1.1}) 
is permanent (Section~\ref{sec:3.1});
a sufficient condition for existence of a solution to system (\ref{1.1})
(Section~\ref{sec:3.2}); and conditions that imply the dynamic system (\ref{1.1})
to have a unique almost periodic solution that is uniformly asymptotically stable	
(Section~\ref{sec:3.3}).


\subsection{Permanence of solutions}
\label{sec:3.1}

We begin by introducing the notion of permanence of solutions.

\begin{definition}
System (\ref{1.1}) is said to be permanent if there exist positive constants $m$ and $M$
such that $ m\leq \liminf_{t \longrightarrow \infty} x_{i}(t)
\leq \limsup_{t \longrightarrow \infty}x_{i}(t)\leq M$, $i=1,2,\ldots,6$,
for any solution $(x_{1}(t),\ldots,x_{6}(t))$ of (\ref{1.1}).
\end{definition}

\begin{theorem}
\label{2}
System (\ref{1.1}) is permanent.
\end{theorem}

\begin{proof}
Let $Z(t)=(x_{1}(t),\ldots,x_{6}(t))$ be a positive solution of system (\ref{1.1}).
Then,
$$
\left\{
\begin{array}{l}
x_{1}^{\Delta}(t)\geq \Lambda -[\lambda^{U}(1-p)+\phi p+\gamma]x_{1}^{\sigma}(t), \\
x_{2}^{\Delta}(t)\geq\lambda^{L}(1-p)m_{1}-(q\nu+\gamma)x_{2}^{\sigma}(t), \\
x_{3}^{\Delta}(t)\geq q\nu m_{2}-(\delta_{1}+\gamma)x_{3}^{\sigma}(t),\\
x_{4}^{\Delta}(t)\geq \phi pm_{1}+\delta_{1}f_{1}m_{3}-(\omega n+\gamma)x_{4}^{\sigma}(t),\\
x_{5}^{\Delta}(t)\geq \delta_{1}(1-f_{1})m_{3}-[\delta_{2}(1-f_{2}-f_{3})
+\delta_{2}f_{2}+\alpha_{1}f_{3}+\gamma]x_{5}^{\sigma}(t),\\
x_{6}^{\Delta}(t)\geq\delta_{2}f_{2}m_{5}-[\eta(1-k)+\alpha_{2}k+\gamma]x^{\sigma}_{6}(t).
\end{array}
\right.
$$
From Lemma~\ref{23}, it follows that
$$
x_{1}(t)\geq \dfrac{\Lambda}{\lambda^{u}(1-p)+\phi p+\gamma}\left[
1+\left( \dfrac{\lambda^{U}(1-p)+\phi p+\gamma}{\Lambda}
x_{1}(t_{0})-1\right) e_{-(\lambda^{U}(1-p)+\phi p+\gamma)}(t,t_{0})\right]
$$
and we have $e_{-(\lambda^{U}(1-p)+\phi p+\gamma)}(t,t_{0})
\longrightarrow 0,\;\; as\;\; t\longrightarrow \infty$.
Thus,
$$
\left\{
\begin{array}{ll}
x_{1}(t)\geq m_{1}:=\dfrac{\Lambda}{\lambda^{U}(1-p)+\phi p+\gamma},
& \text{ for }\, t\geq T_{1}, \\
x_{2}(t)\geq m_{2}:=\dfrac{\lambda^{L}(1-p)m_{1}}{q\nu+\gamma},
& \text{ for }\, t\geq T_{2},\\
x_{3}(t)\geq m_{3}:=\dfrac{q\nu m_{2}}{\delta_{1}+\gamma},
& \text{ for }\,t\geq T_{3},\\
x_{4}(t)\geq m_{4}:=\dfrac{\phi pm_{1}+\delta_{1}f_{1}m_{3}}{\omega n+\gamma},
& \text{ for } \,t\geq T_{4},\\	
x_{5}(t)\geq m_{5}:=\dfrac{\delta_{1}(1-f_{1})m_{3}}{\delta_{2}(1-f_{2}-f_{3})
+\delta_{2}f_{2}+\alpha_{1}f_{3}+\gamma},& \text{ for } \, t\geq T_{5},\\
x_{6}(t)\geq m_{6}:=\dfrac{\delta_{2}f_{2}m_{5}}{\eta(1-k)+\alpha_{2}k+\gamma},
& \text{ for } \,t\geq T_{6}.
\end{array}
\right.
$$
Let $m=\displaystyle\min_{1\leqslant i\leqslant6}{m_{i}}$
and $T=\displaystyle\max_{1\leqslant i\leqslant6}{T_{i}}$.
We can then write that $x_{i}(t)\geq m $ for all $t>T$.	
\end{proof}


\subsection{Existence of solution}
\label{sec:3.2}

For system (\ref{1.1}), we introduce the following assumption:
\begin{itemize}
\item[$(H_{1})$] $\lambda(t)$ is a bounded almost periodic function and satisfy
$0< \lambda^{L}\leq \lambda(t) \leq \lambda^{U}$.
\end{itemize}

To prove existence of solution,
we first begin with a technical lemma.

\begin{lemma}
\label{19}
If $(H1)$ holds, then, for any positive solution
$Z(t)=(x_{1}(t),\ldots,x_{6}(t))$ of system (\ref{1.1}),
there exist positive constants $M$ and $T$
such that $x_{i}(t)<M$, $i=1,\ldots,6$, for all $t>T$.
\end{lemma}

\begin{proof}
Let $Z(t)=(x_{1}(t),\ldots,x_{6}(t))$
be a positive solution of system (\ref{1.1}). Then,
$$
\left\{
\begin{array}{l}
x_{1}^{\Delta}(t)\leqslant \Lambda +\omega n
\frac{\Lambda}{\gamma}-[\lambda^{L}(1-p)+\phi p+\gamma]x_{1}^{\sigma}(t),\\
x_{2}^{\Delta}(t)\leqslant\lambda^{U}(1-p)M_{1}-(q\nu+\gamma)x_{2}^{\sigma}(t),\\
x_{3}^{\Delta}(t)\leqslant q\nu M_{2}-(\delta_{1}+\gamma)x_{3}^{\sigma}(t),\\
x_{4}^{\Delta}(t)\leqslant \phi pM_{1}+\delta_{1}f_{1}M_{3}+\delta_{2}(1-f_{2}
-f_{3})\frac{\Lambda}{\gamma}-(\omega n+\gamma)x_{4}^{\sigma}(t),\\
x_{5}^{\Delta}(t)\leqslant \delta_{1}(1-f_{1})M_{3}+\eta(1-k)\frac{\Lambda}{\gamma}
-[\delta_{2}(1-f_{2}-f_{3})+\delta_{2}f_{2}+\alpha_{1}f_{3}+\gamma]x_{5}^{\sigma}(t),\\
x_{6}^{\Delta}(t)\leqslant\delta_{2}f_{2}M_{5}-[\eta(1-k)+\alpha_{2}k+\gamma]x_{6}^{\sigma}(t).
\end{array}
\right.
$$
From Lemma~\ref{23},
$$
x_{1}(t)\leqslant \dfrac{\Lambda+\omega n
\frac{\Lambda}{\gamma}}{\lambda^{L}(1-p)
+\phi p+\gamma}\left[ 1+\left( \dfrac{\lambda^{L}(1-p)
+\phi p+\gamma}{\Lambda
+\omega n \frac{\Lambda}{\gamma}}x_{1}(t_{0})-1\right)
e_{-(\lambda^{L}(1-p)+\phi p+\gamma)}(t,t_{0})\right]
$$
and we have $e_{-(\lambda^{L}(1-p)+\phi p+\gamma)}(t,t_{0})
\longrightarrow 0,\;\; as\;\; t\longrightarrow \infty$.
Then,
$$
x_{1}(t)\leqslant M_{1}:= \dfrac{\Lambda+\omega n
\frac{\Lambda}{\gamma}}{\lambda^{L}(1-p)+\phi p+\gamma},
$$
$$
\left\{
\begin{array}{ll}
x_{1}(t)\leqslant M_{1}:= \dfrac{\Lambda+\omega
n \frac{\Lambda}{\gamma}}{\lambda^{L}(1-p)+\phi p+\gamma},
& \text{ for } \,t\geq T_{1}, \\
x_{2}(t)\leqslant M_{2}:=\dfrac{\lambda^{U}(1-p)M_{1}}{q\nu+\gamma},
& \text{ for } \,t\geq T_{2},\\
x_{3}(t)\leqslant M_{3}:=\dfrac{q\nu M_{2}}{\delta_{1}+\gamma},
& \text{ for }\, t\geq T_{3},\\
x_{4}(t)\leqslant M_{4}:=\dfrac{\phi pM_{1}+\delta_{1}f_{1}M_{3}+\delta_{2}(1-f_{2}-f_{3})
\frac{\Lambda}{\gamma}}{\omega n+\gamma},
& \text{ for }\,t\geq T_{4},\\
x_{5}(t)\leqslant M_{5}:=\dfrac{\delta_{1}(1-f_{1})M_{3}+\eta(1-k)
\frac{\Lambda}{\gamma}}{\delta_{2}(1-f_{2}-f_{3})+\delta_{2}f_{2}+\alpha_{1}f_{3}+\gamma},
& \text{ for } \,t\geq T_{5},\\
x_{6}(t)\leqslant M_{6}:=\dfrac{\delta_{2}f_{2}M_{5}}{\eta(1-k)+\alpha_{2}k+\gamma},
& \text{ for } \,t\geq T_{6}.
\end{array}
\right.
$$
Let $M=\displaystyle\max_{1\leqslant i\leqslant6}{M_{i}}$
and $T=\displaystyle\max_{1\leqslant i\leqslant6}{T_{i}}$.
Then $x_{i}(t)\leqslant M $ for all $t>T$.
\end{proof}

Define
\begin{eqnarray*}
\Omega&=&\lbrace (x_{1}(t),\ldots,x_{6}(t)):\,
(x_{1}(t),\ldots,x_{6}(t))\\
& &\textit{ is a solution of (\ref{1.1}) and }\,
0 <m \leq x_{i}\leq M,\;i=1,\ldots,6\rbrace.
\end{eqnarray*}

\begin{theorem}
\label{thm:existence}
If $(H1)$ holds, then the set $\Omega$ is nonempty.
\end{theorem}

\begin{proof}
The almost periodicity of $\lambda(t)$ implies that there is a sequence
$\lbrace\xi_{l}\rbrace\subseteq\mathbb{T}^{+}$ with $l\rightarrow\infty$
such that
$$
\lambda(t+\xi_{l})\rightarrow\lambda(t)\;\; \text{as} \;\;l\rightarrow\infty.
$$
From Theorem~\ref{2} and Lemma~\ref{19}, for each sufficiently small $\varepsilon> 0$,
there exists a $t_{1}\in\mathbb{T}^{+}$ such that
$$
m-\varepsilon\leq x_{i}(t)\leq M+\varepsilon,\,
\text{ for all }\,t\geq t_{1},
\quad i=1,\ldots, 6.
$$
Set $x_{il}(t)=x_{i}(t+\xi_{l})$ for $t\geq t_{1}-\xi_{l},\,l=1,2,\ldots$
For any positive integer $k$, we obtain that there exists a sequence
$\{x_{il}(t):\,l\geq k\}$ such that the sequence $\{x_{il}(t)\}$
has a subsequence, denoted by $\{x_{il}^{*}(t)\}$ $(x_{il}^{*}(t)=x_{il}(t+\xi_{l}^{*}))$,
converging on any finite interval of $\mathbb{T}^{+}$ as $l\rightarrow\infty$.
So we have a sequence $\{y_{i}(t)\}$ such that, for $t\in\mathbb{T}^{+}$,
\begin{equation}
\label{e1}
x_{il}^{*}(t)\longrightarrow y_{i}(t),\;
\mbox{ as }l\rightarrow\infty,\;i=1,\ldots, 6.
\end{equation}

It is easy to see that the above sequence $\{\xi_{l}^{*}\}\subseteq\mathbb{T}^{+}$
with $\xi_{l}^{*}\rightarrow\tau$ for $l\rightarrow\infty$ is such that
$$
\lambda(t+\xi_{l}^{*})\longrightarrow\lambda(t),\;
\text{ as }\; l\rightarrow\infty,
$$
which, together with (\ref{e1}) and
$$
\left\{
\begin{array}{l}
x^{*\Delta}_{1l}(t)=\Lambda+\omega nx^{*}_{4l}(t)
-[\lambda(t+\xi_{l}^{*})(1-p)+\phi p+\gamma]x_{1l}^{*\sigma}(t),\\
x^{*\Delta}_{2l}(t)=\lambda(t+\xi_{l}^{*}))(1-p)
x_{1l}^{*}(t)-[q\nu+\gamma]x_{2l}^{*\sigma}(t),\\
x^{*\Delta}_{3l}(t)=q\nu x_{2l}^{*}(t)-[\delta_{1}+\gamma]x_{3l}^{*\sigma}(t),\\
x^{*\Delta}_{4l}(t)=\phi px_{1l}^{*}(t)+\delta_{1}f_{1}
x_{3l}^{*}(t)+\delta_{2}(1-f_{2}-f_{3})x_{5l}^{*}(t)
-[\omega n+\gamma]x^{*\sigma}_{4l}(t),\\
x^{*\Delta}_{5l}(t)=\delta_{1}(1-f_{1})x_{3l}^{*}(t)
+\eta(1-k)x_{6l}^{*}(t)-[\delta_{2}(1-f_{2}-f_{3})
+\delta_{2}f_{2}+\alpha_{1}f_{3}+\gamma]x_{5l}^{*\sigma}(t),\\
x^{*\Delta}_{6l}(t)=\delta_{2}f_{2}x_{5l}^{*}(t)
-[\eta(1-k)+\alpha_{2}k+\gamma]x_{6l}^{*\sigma}(t),
\end{array}
\right.
$$
yields
$$
\left\{
\begin{array}{l}
y^{\Delta}_{1}(t)=\Lambda+\omega ny_{4}(t)
-[\lambda(t)(1-p)+\phi p+\gamma]y_{1}^{\sigma}(t),\\
y^{\Delta}_{2}(t)=\lambda(t)(1-p)y_{1}(t)
-[q\nu+\gamma]y_{2}^{\sigma}(t),\\
y^{\Delta}_{3}(t)=q\nu y_{2}(t)-[\delta_{1}
+\gamma]y_{3}^{\sigma}(t),\\
y^{\Delta}_{4}(t)=\phi py_{1}(t)+\delta_{1}f_{1}y_{3}(t)
+\delta_{2}(1-f_{2}-f_{3})y_{5}(t)-[\omega n+\gamma]y^{\sigma}_{4}(t),\\
y^{\Delta}_{5}(t)=\delta_{1}(1-f_{1})y_{3}(t)+\eta(1-k)y_{6}(t)
-\left[\delta_{2}(1-f_{2}-f_{3})+\delta_{2}f_{2}+\alpha_{1}f_{3}+\gamma\right]
y_{5}^{\sigma}(t),\\
y^{\Delta}_{6}(t)=\delta_{2}f_{2}y_{5}(t)
-\left[\eta(1-k)+\alpha_{2}k+\gamma\right]
y_{6}^{\sigma}(t).
\end{array}
\right.
$$
It is clear that $(y_{1}(t),\ldots,y_{6}(t))$
is a solution of system (\ref{1.1}) and
$$
m-\varepsilon\leq y_{i}(t)\leq M+\varepsilon,\,
\text{ for all }\,t\in\mathbb{T}^{+},
\quad i=1,\ldots, 6.
$$
Since $\varepsilon$ is arbitrary, it follows that
$$
m\leq y_{i}(t)\leq M,\,\text{ for }\,t\in\mathbb{T}^{+},
\quad i=1,\ldots, 6.
$$
The proof is complete.
\end{proof}


\subsection{Uniform asymptotic stability}
\label{sec:3.3}

Now, we establish sufficient conditions
for the existence of a unique positive
almost periodic solution to system (\ref{1.1})
that is uniform asymptotically stable.
We introduce some more notations. Let
\begin{eqnarray*}
A_{1}&:=&\lambda^{L}(1-p)+\phi p+\gamma;
\quad A_{2}:=q\nu+\gamma;\\
A_{3}&:=&\delta_{1}+\gamma;\quad A_{4}:=\omega n+\gamma;\\
A_{5}&:=&\delta_{2}(1-f_{3})+\alpha_{1} f_{3}+\gamma;\quad 
A_{6}:=\eta(1-k)+\alpha_{2}k +\gamma;\\
B_{1}&:=&\lambda^{U}(1-p)+\phi p ;
\quad B_{2}:= q\nu+2\frac{\gamma\beta l_{A}(1-p)M}{\Lambda};\\
B_{3}&:=&\delta_{1}+2\frac{\gamma\beta (1-p)M}{\Lambda};
\quad B_{4}:=\omega n;\\
B_{5}&:=&\delta_{2}(1-f_{3})+2 \frac{\gamma\beta l_{H}(1-p)M}{\Lambda};
\quad B_{6}:=\eta(1-k).
\end{eqnarray*}
Moreover, let $A:=\displaystyle\min_{1\leq i\leq 6} A_{i}$
and $B:=\displaystyle\max_{1\leq i\leq 6}B_{i}$.

In our next result (Theorem~\ref{impth}), 
we assume the following additional hypothesis:
\begin{itemize}
\item[(H2)] $B<A$.
\end{itemize}

\begin{theorem}
\label{impth}
If (H1) and (H2) hold, then the dynamic system (\ref{1.1})
has a unique almost periodic solution
$Z(t)=(x_{1}(t),\ldots,x_{6}(t))\in \Omega$
that is uniformly asymptotically stable.	
\end{theorem}

\begin{proof}
According to Theorem~\ref{2}, every solution $Z(t)=(x_{1}(t),\ldots,x_{6}(t))$
of system (\ref{1.1}) satisfies $x_{i}^{L}\leq x_{i}\leq x_{i}^{U}$. Hence,
$\left\vert x_{i}(t)\right\vert\leq K_{i},\;\;i=1,\ldots, 6$.
Suppose that $Z(t)=(x_{1}(t),\ldots,x_{6}(t))$
and $\widehat{Z}(t)=(\widehat{x}_{1}(t),\ldots,\widehat{x}_{6}(t))$
are two positive solutions of system (\ref{1.1}). We have
$$
\left\{
\begin{array}{l}
x^{\Delta}_{1}(t)=\Lambda+\omega nx_{4}(t)-[\lambda(t)(1-p)+\phi p+\gamma]x_{1}^{\sigma}(t),\\
x^{\Delta}_{2}(t)=\lambda(t)(1-p)x_{1}(t)-[q\nu+\gamma]x_{2}^{\sigma}(t),\\
x^{\Delta}_{3}(t)=q\nu x_{2}(t)-[\delta_{1}+\gamma]x_{3}^{\sigma}(t),\\
x^{\Delta}_{4}(t)=\phi px_{1}(t)+\delta_{1}f_{1}x_{3}(t)+\delta_{2}(1-f_{2}
-f_{3})x_{5}(t)-[\omega n+\gamma]x^{\sigma}_{4}(t),\\
x^{\Delta}_{5}(t)=\delta_{1}(1-f_{1})x_{3}(t)+\eta(1-k)x_{6}(t)-[\delta_{2}(1
-f_{2}-f_{3})+\delta_{2}f_{2}+\alpha_{1}f_{3}+\gamma]x_{5}^{\sigma}(t),\\
x^{\Delta}_{6}(t)=\delta_{2}f_{2}x_{5}(t)-[\eta(1-k)+\alpha_{2}k+\gamma]x_{6}^{\sigma}(t),
\end{array}
\right.
$$
and
$$
\left\{
\begin{array}{l}
\widehat{x}^{\Delta}_{1}(t)=\Lambda+\omega n\widehat{x}_{4}(t)
-[\widehat{\lambda}(t)(1-p)+\phi p+\gamma]\widehat{x}_{1}^{\sigma}(t), \\
\widehat{x}^{\Delta}_{2}(t)=\widehat{\lambda}(t)(1-p)\widehat{x}_{1}(t)
-[q\nu+\gamma]\widehat{x}_{2}^{\sigma}(t),\\
\widehat{x}^{\Delta}_{3}(t)=q\nu \widehat{x}_{2}(t)
-[\delta_{1}+\gamma]\widehat{x}^{\sigma}_{3}(t),\\
\widehat{x}^{\Delta}_{4}(t)=\phi p\widehat{x}_{1}(t)+\delta_{1}
f_{1}\widehat{x}_{3}(t)+\delta_{2}(1-f_{2}-f_{3})
\widehat{x}_{5}(t)-[\omega n+\gamma]\widehat{x}_{4}^{\sigma}(t),\\
\widehat{x}^{\Delta}_{5}(t)=\delta_{1}(1-f_{1})\widehat{x}_{3}(t)+\eta(1-k)
\widehat{x}_{6}(t)-[\delta_{2}(1-f_{2}-f_{3})+\delta_{2}f_{2}
+\alpha_{1}f_{3}+\gamma]\widehat{x}_{5}^{\sigma}(t),\\
\widehat{x}^{\Delta}_{6}(t)=\delta_{2}f_{2}\widehat{x}_{5}(t)
-[\eta(1-k)+\alpha_{2}k+\gamma]\widehat{x}_{6}^{\sigma}(t).
\end{array}
\right.
$$
Denote
$$
\left\Vert Z\right\Vert=\left\Vert (x_{1}(t),\ldots,x_{6}(t)) \right\Vert
=\displaystyle\sup_{t\in\mathbb{T}^{+}}
\displaystyle\sum_{i=1}^{6}\left\vert x_{i}(t)\right\vert.
$$
Then
$\left\Vert Z\right\Vert\leq K$ and $\left\Vert \widehat{Z}\right\Vert\leq K$
where $K=\displaystyle\sum_{i=1}^{6}K_{i}$.
Define the Lyapunov function $V(t,Z,\widehat{Z})$
on $\mathbb{T}^{+}\times\Omega\times\Omega$ as
\begin{equation}
\label{lyaponov}
V(t,Z,\widehat{Z})=\displaystyle\sum^{6}_{i=1}\left\vert x_{i}(t)
-\widehat{x}_{i}(t)\right\vert=\displaystyle\sum^{6}_{i=1}V_{i}(t),
\end{equation}
where $V_{i}(t)=\left\vert x_{i}(t)-\widehat{x}_{i}(t)\right\vert$.
Then the two norms
$$
\left\Vert Z-\widehat{Z}\right\Vert
=\displaystyle\sum^{6}_{i=1}\left\vert x_{i}(t)
-\widehat{x}_{i}(t) \right\vert
$$
and
$$
\left\Vert Z-\widehat{Z}\right\Vert_{*}=\displaystyle\sup_{t\in\mathbb{R}_{+}}
\left[\displaystyle\sum^{6}_{i=1}\left(x_{i}(t)
-\widehat{x}_{i}(t) \right)^{2}\right]^{\frac{1}{2}}
$$
are equivalent, that is, there exist two constants
$\eta_{1},\,\eta_{2}>0$ such that
$$
\eta_{1}\left\Vert Z-\widehat{Z}\right\Vert_{*}\leq\left\Vert Z
-\widehat{Z}\right\Vert\leq\eta_{2}\left\Vert Z-\widehat{Z}\right\Vert_{*}.
$$
Hence,
$$
\eta_{1}\left\Vert Z-\widehat{Z}\right\Vert_{*}
\leq V(t, Z,\widehat{Z})\leq\eta_{2}\left\Vert Z-\widehat{Z}\right\Vert_{*}.
$$
Let $a,\,b\in C(\mathbb{R}^{+},\mathbb{R}^{+})$, $a(x)=\eta_{1}x$, and $b(x)=\eta_{2} x$.
Then the assumption (i) of Lemma~\ref{24} is satisfied.
On the other hand, we have
\begin{eqnarray*}
\left\vert V(t,Z(t),\widehat{Z}(t))-V(t,Z^{*}(t),\widehat{Z}^{*}(t))\right\vert
&=&\left\vert\displaystyle\sum^{6}_{i=1}\left\vert x_{i}(t)-\widehat{x}_{i}(t)
\right\vert-\displaystyle\sum^{6}_{i=1}\left\vert x_{i}^{*}(t)
-\widehat{x}_{i}^{*}(t)\right\vert\right\vert\\
&\leq&\displaystyle\sum^{6}_{i=1}\left\vert\left\vert x_{i}(t)
-\widehat{x}_{i}(t)\right\vert-\left\vert x_{i}^{*}(t)
-\widehat{x}_{i}^{*}(t)\right\vert\right\vert\\
&\leq& \displaystyle\sum^{6}_{i=1}\left\vert x_{i}(t)-x^{*}_{i}(t)\right\vert
+\displaystyle\sum^{6}_{i=1}\left\vert \widehat{x}_{i}(t)
-\widehat{x}_{i}^{*}(t)\right\vert\\
&\leq&\left\Vert Z-Z^{*}\right\Vert+\left\Vert \widehat{Z}
-\widehat{Z}^{*}\right\Vert,
\end{eqnarray*}
where $L=1$, so condition (ii) of Lemma~\ref{24} is also satisfied.

Now, using Lemma~4.2 of \cite{[wong]}, it follows that
$$
D^{+}V^{\Delta}_{i}(t)\leq sign(x_{i}^{\sigma}(t)
-\widehat{x}_{i}^{\sigma}(t))(x_{i}^{\Delta}(t)
-\widehat{x}_{i}^{\Delta}(t)), \;\; i=1,\ldots,6,
$$
where $D^{+}V^{\Delta}_{i}$ is the Dini derivative
of $V_{i}$. For $i=1$,
\begin{eqnarray*}
D^{+}V^{\Delta}_{1}(t)&\leq&sign(x_{1}^{\sigma}(t)
-\widehat{x}_{1}^{\sigma}(t))(x_{1}^{\Delta}(t)
-\widehat{x}_{1}^{\Delta}(t))\\
&=&sign(x_{1}^{\sigma}(t)-\widehat{x}_{1}^{\sigma}(t))[
\omega n(x_{4}(t)-\widehat{x}_{4}(t))-\lambda(t)(1-p)x_{1}^{\sigma}(t)\\
& &-[\phi p+\gamma](x_{1}^{\sigma}(t)-\widehat{x}^{\sigma}_{1}(t))
+\widehat{\lambda}(t)(1-p)\widehat{x}^{\sigma}_{1}(t)]\\
&=&sign(x_{1}^{\sigma}(t)-\widehat{x}_{1}^{\sigma}(t))[\omega n(x_{4}(t)
-\widehat{x}_{4}(t))-\lambda(t)(1-p)x_{1}^{\sigma}(t)\\
& &-[\phi p+\gamma](x_{1}^{\sigma}(t)-\widehat{x}_{1}^{\sigma}(t))
+\widehat{\lambda}(t)(1-p)\widehat{x}^{\sigma}_{1}(t)
+\lambda(t)(1-p)\widehat{x}^{\sigma}_{1}(t)
-\lambda(t)(1-p)\widehat{x}_{1}^{\sigma}(t)]\\
&=&sign(x_{1}^{\sigma}(t)-\widehat{x}_{1}^{\sigma}(t))
[\omega n(x_{4}(t)-\widehat{x}_{4}(t))\\
& &-[\lambda(t)(1-p)+\phi p+\gamma](x_{1}^{\sigma}(t)-\widehat{x}_{1}^{\sigma}(t))\\
& &+(\widehat{\lambda}(t)-\lambda(t))(1-p)\widehat{x}_{1}^{\sigma}(t)]\\
&=&sign(x_{1}^{\sigma}(t)-\widehat{x}_{1}^{\sigma}(t))[\omega n(x_{4}(t)
-\widehat{x}_{4}(t))-[\lambda(t)(1-p)+\phi p+\gamma]
(x_{1}^{\sigma}(t)-\widehat{x}_{1}^{\sigma}(t))\\
& &+(\widehat{\lambda}(t)-\lambda(t))(1-p)\widehat{x}^{\sigma}_{1}(t)]\\
&\leq&\omega n\left\vert x_{4}(t)-\widehat{x}_{4}(t)\right\vert
-[\lambda^{L}(1-p)+\phi p+\gamma]\left\vert x_{1}^{\sigma}(t)
-\widehat{x}_{1}^{\sigma}(t)\right\vert\\
& &+\frac{\gamma\beta l_{A}(1-p)M}{\Lambda}\left\vert x_{2}(t)
-\widehat{x}_{2}(t)\right\vert+\frac{\gamma\beta (1-p)M}{\Lambda}
\left\vert x_{3}(t)-\widehat{x}_{3}(t)\right\vert\\
& &+\frac{\gamma\beta l_{H}(1-p)M}{\Lambda}\left\vert x_{5}(t)
-\widehat{x}_{5}(t)\right\vert;
\end{eqnarray*}
for $i=2$,
\begin{eqnarray*}
D^{+}V^{\Delta}_{2}(t)
&\leq&sign(x_{2}^{\sigma}(t)-\widehat{x}_{2}^{\sigma}(t))(x_{2}^{\Delta}(t)
-\widehat{x}_{2}^{\Delta}(t))\\
&=&sign(x_{2}^{\sigma}(t)-\widehat{x}_{2}^{\sigma}(t))[\lambda(t)(1-p)x_{1}(t)
-[q\nu+\gamma](x_{2}^{\sigma}(t)-\widehat{x}_{2}^{\sigma}(t))
-\widehat{\lambda}(t)(1-p)\widehat{x}_{1}(t)]\\
&=&sign(x_{2}^{\sigma}(t)-\widehat{x}_{2}^{\sigma}(t))[\lambda(t)(1-p)x_{1}(t)
-[q\nu+\gamma](x_{2}^{\sigma}(t)-\widehat{x}_{2}^{\sigma}(t))\\
& &-\widehat{\lambda}(t)(1-p)\widehat{x}_{1}(t)
+\lambda(t)(1-p)\widehat{x}_{1}(t)-\lambda(t)(1-p)\widehat{x}_{1}(t)]\\
&=&sign(x_{2}^{\sigma}(t)-\widehat{x}_{2}^{\sigma}(t))
[-[q\nu+\gamma](x_{2}^{\sigma}(t)-\widehat{x}_{2}^{\sigma}(t))\\
& &+\lambda(t)(1-p)(x_{1}(t)-\widehat{x}_{1}(t))-(\widehat{\lambda}(t)
-\lambda(t))(1-p)\widehat{x}_{1}(t)]\\
&\leq&\lambda^{U}(1-p)\left\vert x_{1}(t)-\widehat{x}_{1}(t)\right\vert
-[q\nu+\gamma]\left\vert x_{2}^{\sigma}(t)-\widehat{x}_{2}^{\sigma}(t)\right\vert \\
& &+(1-p)M\left\vert\lambda(t)-\widehat{\lambda}(t)\right\vert\\
&\leq&\lambda^{U}(1-p)\left\vert x_{1}(t)-\widehat{x}_{1}(t)\right\vert
-[q\nu+\gamma]\left\vert x_{2}^{\sigma}(t)-\widehat{x}_{2}^{\sigma}(t)\right\vert \\
& &+\frac{\gamma\beta l_{A}(1-p)M}{\Lambda}\left\vert x_{2}(t)-\widehat{x}_{2}(t)\right\vert\\
& &+\frac{\gamma\beta (1-p)M}{\Lambda}\left\vert x_{3}(t)-\widehat{x}_{3}(t)
\right\vert+\frac{\gamma\beta l_{H}(1-p)M}{\Lambda}\left\vert x_{5}(t)-\widehat{x}_{5}(t)\right\vert;
\end{eqnarray*}
for $i=3$,
\begin{eqnarray*}
D^{+}V^{\Delta}_{3}(t)&\leq&sign(x_{3}^{\sigma}(t)
-\widehat{x}_{3}^{\sigma}(t))(x_{3}^{\Delta}(t)-\widehat{x}_{3}^{\Delta}(t))\\
&=&sign(x_{3}^{\sigma}(t)-\widehat{x}_{3}^{\sigma}(t))[q\nu(x_{2}(t)-\widehat{x}_{2}(t))\\
& &-[\delta_{1}+\gamma](x_{3}^{\sigma}(t)-\widehat{x}_{3}^{\sigma}(t))]\\
&\leq&q\nu\left\vert x_{2}(t)-\widehat{x}_{2}(t)\right\vert-[\delta_{1}
+\gamma]\left\vert x_{3}^{\sigma}(t)-\widehat{x}_{3}^{\sigma}(t)\right\vert;
\end{eqnarray*}
for $i=4$,
\begin{eqnarray*}
D^{+}V^{\Delta}_{4}(t)
&\leq&sign(x_{4}^{\sigma}(t)-\widehat{x}_{4}^{\sigma}(t))(x_{4}^{\Delta}(t)
-\widehat{x}_{4}^{\Delta}(t))\\
&=&sign(x_{4}^{\sigma}(t)-\widehat{x}_{4}^{\sigma}(t))[\phi p(x_{1}(t)
-\widehat{x}_{1}(t))+\delta_{1}f_{1}(x_{3}(t)-\widehat{x}_{3}(t))\\
& &+\delta_{2}(1-f_{2}-f_{3})(x_{5}(t)-\widehat{x}_{5}(t))
-[\omega n+\gamma](x_{4}^{\sigma}(t)-\widehat{x}_{4}^{\sigma}(t))]\\
&\leq&\phi p\left\vert x_{1}(t)-\widehat{x}_{1}(t)\right\vert
+\delta_{1}f_{1}\left\vert x_{3}(t)-\widehat{x}_{3}(t)\right\vert\\
& &+\delta_{2}(1-f_{2}-f_{3})\left\vert x_{5}(t)-\widehat{x}_{5}(t)\right\vert
-[\omega n+\gamma]\left\vert x_{4}^{\sigma}(t)-\widehat{x}_{4}^{\sigma}(t)\right\vert;
\end{eqnarray*}
for $i=5$,
\begin{eqnarray*}
D^{+}V^{\Delta}_{5}(t)&\leq&sign(x_{5}^{\sigma}(t)
-\widehat{x}_{5}^{\sigma}(t))(x_{5}^{\Delta}(t)-\widehat{x}_{5}^{\Delta}(t))\\
&=&sign(x_{5}^{\sigma}(t)
-\widehat{x}_{5}^{\sigma}(t))[\delta_{1}(1-f_{1})(x_{3}(t)
-\widehat{x}_{3}(t))+\eta(1-k)(x_{6}(t)-\widehat{x}_{6}(t))\\
& &-[\delta_{2}(1-f_{2}-f_{3})+\delta_{1}f_{1}+\alpha_1 f_{3}
+\gamma](x_{5}^{\sigma}(t)-\widehat{x}_{5}^{\sigma}(t))]\\
&\leq&\delta_{1}(1-f_{1})\left\vert x_{3}(t)-\widehat{x}_{3}(t)\right\vert
+\eta(1-k)\left\vert x_{6}(t)-\widehat{x}_{6}(t)\right\vert\\
& &-[\delta_{2}(1-f_{2}-f_{3})+\delta_{1}f_{1}+\alpha_1 f_{3}
+\gamma]\left\vert x_{5}^{\sigma}(t)-\widehat{x}_{5}^{\sigma}(t)\right\vert;
\end{eqnarray*}
and, finally, for $i=6$,
\begin{eqnarray*}
D^{+}V^{\Delta}_{6}(t)&\leq&sign(x_{6}^{\sigma}(t)
-\widehat{x}_{6}^{\sigma}(t))(x_{6}^{\Delta}(t)-\widehat{x}_{6}^{\Delta}(t))\\
&=&sign(x_{6}^{\sigma}(t)
-\widehat{x}_{6}^{\sigma}(t))[\delta_{2}f_{2}(x_{5}(t)-\widehat{x}_{5}(t))\\
& &-[\eta(1-k)+\alpha_{2}k+\gamma](x_{6}^{\sigma}(t)
-\widehat{x}_{6}^{\sigma}(t)\\
&\leq&\delta_{2}f_{2}\left\vert x_{5}(t)-\widehat{x}_{5}(t)\right\vert
-[\eta(1-k)+\alpha_{2}k +\gamma]\left\vert x_{6}^{\sigma}(t)
-\widehat{x}_{6}^{\sigma}(t)\right\vert.
\end{eqnarray*}
It follows that
\begin{eqnarray*}
D^{+}V^{\Delta}(t)&\leq&\omega n\left\vert x_{4}(t)-\widehat{x}_{4}(t)\right\vert
-[\lambda^{L}(1-p)+\phi p+\gamma]\left\vert x_{1}^{\sigma}(t)
-\widehat{x}_{1}^{\sigma}(t)\right\vert\\
& &+\frac{\gamma\beta l_{A}(1-p)M}{\Lambda}\left\vert x_{2}(t)
-\widehat{x}_{2}(t)\right\vert+\frac{\gamma\beta (1-p)M}{\Lambda}\left\vert x_{3}(t)
-\widehat{x}_{3}(t)\right\vert\\
& &+\frac{\gamma\beta l_{H}(1-p)M}{\Lambda}\left\vert x_{5}(t)
-\widehat{x}_{5}(t)\right\vert+\lambda^{U}(1-p)\left\vert x_{1}(t)
-\widehat{x}_{1}(t)\right\vert\\
& &-[q\nu+\gamma]\left\vert x_{2}^{\sigma}(t)-\widehat{x}_{2}^{\sigma}(t)\right\vert
+\frac{\gamma\beta l_{A}(1-p)M}{\Lambda}\left\vert x_{2}(t)-\widehat{x}_{2}(t)\right\vert\\
& &+\frac{\gamma\beta (1-p)M}{\Lambda}\left\vert x_{3}(t)
-\widehat{x}_{3}(t)\right\vert+\frac{\gamma\beta l_{H}(1-p)M}{\Lambda}\left\vert x_{5}(t)
-\widehat{x}_{5}(t)\right\vert\\
& &+q\nu\left\vert x_{2}(t)-\widehat{x}_{2}(t)\right\vert
-[\delta_{1}+\gamma]\left\vert x_{3}^{\sigma}(t)-\widehat{x}_{3}^{\sigma}(t)\right\vert\\
& &+\phi p\left\vert x_{1}(t)-\widehat{x}_{1}(t)\right\vert
+\delta_{1}f_{1}\left\vert x_{3}(t)-\widehat{x}_{3}(t)\right\vert\\
& &+\delta_{2}(1-f_{2}-f_{3})\left\vert x_{5}(t)-\widehat{x}_{5}(t)\right\vert
-[\omega n+\gamma]\left\vert x_{4}^{\sigma}(t)-\widehat{x}_{4}^{\sigma}(t)\right\vert\\
& &+\delta_{1}(1-f_{1})\left\vert x_{3}(t)-\widehat{x}_{3}(t)\right\vert
+\eta(1-k)\left\vert x_{6}(t)-\widehat{x}_{6}(t)\right\vert\\
& &-[\delta_{2}(1-f_{2}-f_{3})+\delta_{1}f_{1}+\alpha_1 f_{3}
+\gamma]\left\vert x_{5}^{\sigma}(t)-\widehat{x}_{5}^{\sigma}(t)\right\vert\\
& &+\delta_{2}f_{2}\left\vert x_{5}(t)-\widehat{x}_{5}(t)\right\vert
-[\eta(1-k)+\alpha_{2}k +\gamma]\left\vert x_{6}^{\sigma}(t)
-\widehat{x}_{6}^{\sigma}(t)\right\vert\\
&=&-[\lambda(t)(1-p)+\phi p+\gamma]\left\vert x_{1}^{\sigma}(t)
-\widehat{x}_{1}^{\sigma}(t)\right\vert-[q\nu+\gamma]\left\vert x_{2}^{\sigma}(t)
-\widehat{x}_{2}^{\sigma}(t)\right\vert\\
& &-[\delta_{1}+\gamma]\left\vert x_{3}^{\sigma}(t)
-\widehat{x}_{3}^{\sigma}(t)\right\vert-[\omega n+\gamma]\left\vert
x_{4}^{\sigma}(t)-\widehat{x}_{4}^{\sigma}(t)\right\vert\\
& &-[\delta_{2}(1-f_{2}-f_{3})+\delta_{1}f_{1}+\alpha_1 f_{3}+\gamma]
\left\vert x_{5}^{\sigma}(t)-\widehat{x}_{5}^{\sigma}(t)\right\vert\\
& &-[\eta(1-k)+\alpha_{2}k +\gamma]\left\vert x_{6}^{\sigma}(t)
-\widehat{x}_{6}^{\sigma}(t)\right\vert\\
& &+\left\lbrace\lambda^{U}(1-p)+\phi p\right\rbrace\left\vert x_{1}(t)
-\widehat{x}_{1}(t)\right\vert\\
& &+\left\lbrace q\nu+\frac{\gamma\beta l_{A}(1-p)M}{\Lambda}+\frac{\gamma\beta
l_{A}(1-p)M}{\Lambda}\right\rbrace\left\vert x_{2}(t)-\widehat{x}_{2}(t)\right\vert\\
& &+\left\lbrace\delta_{1}f_{1}+\delta_{1}(1-f_{1})
+\frac{\gamma\beta (1-p)M}{\Lambda}+\frac{\gamma\beta (1-p)M}{\Lambda}\right\rbrace
\left\vert x_{3}(t)-\widehat{x}_{3}(t)\right\vert\\
& &+\omega n \left\vert x_{4}(t)-\widehat{x}_{4}(t)\right\vert
+\left\lbrace\delta_{2}(1-f_{2}-f_{3})+\delta_{2}f_{2}
+\frac{\gamma\beta l_{H}(1-p)M}{\Lambda}\right. \\
& &\left. +\frac{\gamma\beta l_{H}(1-p)M}{\Lambda}\right\rbrace\left\vert x_{5}(t)
-\widehat{x}_{5}(t)\right\vert+ \eta(1-k) \left\vert x_{6}(t)-\widehat{x}_{6}(t)\right\vert\\
&=&-A_{1}\left\vert x_{1}^{\sigma}(t)-\widehat{x}_{1}^{\sigma}(t)\right\vert
-A_{2}\left\vert x_{2}^{\sigma}(t)-\widehat{x}_{2}^{\sigma}(t)\right\vert\\
& &-A_{3}\left\vert x_{3}^{\sigma}(t)-\widehat{x}_{3}^{\sigma}(t)\right\vert
-A_{4}\left\vert x_{4}^{\sigma}(t)-\widehat{x}_{4}^{\sigma}(t)\right\vert\\
& &-A_{5}\left\vert x_{5}^{\sigma}(t)-\widehat{x}_{5}^{\sigma}(t)\right\vert
-A_{6}\left\vert x_{6}^{\sigma}(t)-\widehat{x}_{6}^{\sigma}(t)\right\vert\\
& &+B_{1}\left\vert x_{1}(t)-\widehat{x}_{1}(t)\right\vert
+B_{2}\left\vert x_{2}(t)-\widehat{x}_{2}(t)\right\vert\\
& &+B_{3}\left\vert x_{3}(t)-\widehat{x}_{3}(t)\right\vert
+B_{4}\left\vert x_{4}(t)-\widehat{x}_{4}(t)\right\vert\\
& &+B_{5}\left\vert x_{5}(t)-\widehat{x}_{5}(t)\right\vert
+B_{6}\left\vert x_{6}(t)-\widehat{x}_{6}(t)\right\vert\\
&=&-AV(\sigma(t))+BV(t)\\
&=&(B-A)V(t)-A\mu(t)D^{+}V^{\Delta}(t)
\end{eqnarray*}
and $D^{+}V^{\Delta}(t)\leq\frac{B-A}{1+A\mu(t)}V(t)\leq-\psi(t) V(t)$ 
with $\psi=\frac{A-B}{1+A \mu^{U}}$.
By $(H2)$, we have $\psi(t)=\frac{A-B}{1+A\mu^{U}}>0$ 
and $1-\psi\mu(t) =1+A(\mu^{U}-\mu(t))+\mu(t) B >0$. 
Hence,  $-\psi\in\mathcal{R}^{+}$.
Thus, the assumption (iii) of Lemma~\ref{24} 
is satisfied and it follows from
Lemma~\ref{24} that there exists a unique almost
periodic solution $Z(t)=(x_{1}(t),\ldots,x_{6}(t))$
of the dynamic system (\ref{1.1}) that is uniformly
asymptotically stable with $Z(t)\in\Omega$.
\end{proof}

We illustrate our results with an example.

\begin{ex}
\label{ex01}
Based on \cite{[b5]}, let us consider the following system on the time scale 
$\mathbb{T}=\mathbb{Z}_0^{+}$:
\begin{equation}
\label{ex}
\left\{
\begin{array}{l}
x^{\Delta}_{1}(t)=\Lambda+\omega nx_{4}(t)
-[\lambda(t)(1-p)+\phi p+\gamma]x_{1}^{\sigma}(t),\\
x^{\Delta}_{2}(t)=\lambda(t)(1-p)x_{1}(t)-[q\nu+\gamma]x_{2}^{\sigma}(t),\\
x^{\Delta}_{3}(t)=q\nu x_{2}(t)-[\delta_{1}+\gamma]x_{3}^{\sigma}(t),\\
x^{\Delta}_{4}(t)=\phi px_{1}(t)+\delta_{1}f_{1}x_{3}(t)+\delta_{2}(1-f_{2}
-f_{3})x_{5}(t)-[\omega n+\gamma]x^{\sigma}_{4}(t),\\
x^{\Delta}_{5}(t)=\delta_{1}(1-f_{1})x_{3}(t)+\eta(1-k)x_{6}(t)-[\delta_{2}(1
-f_{2}-f_{3})+\delta_{2}f_{2}+\alpha_{1}f_{3}+\gamma]x_{5}^{\sigma}(t),\\
x^{\Delta}_{6}(t)=\delta_{2}f_{2}x_{5}(t)
-[\eta(1-k)+\alpha_{2}k+\gamma]x_{6}^{\sigma}(t),
\end{array}
\right.
\end{equation}
subject to 
$$
x_1(0)=10283785, \quad
x_2(0)=13, \quad
x_3(0)=2, \quad
x_4(0)=0,\quad
x_5(0)=0, \quad 
x_6(0)=0,
$$
where $\Lambda=\frac{22614}{53}$, $\omega=1/31$, $n=0.075$,
$$
\lambda(t)=\dfrac{\beta\left(l_{A}x_2(t)+x_3(t)+l_{H}x_5(t)\right)}{N(t)}
$$
with $\beta=1.93$, $l_A=1$, $l_H=0.1$, and $N(t)=\sum_{i=1}^{6} x_i(t)$,
$p=0.68$, $\phi=1/12$, $\gamma=\frac{47833615}{N_0}$ with $N_0=N(0)$,
$q=0.15$, $\nu=1/15$, $\delta_1=1/3$, $\delta_2=1/3$,
$f_1=0.96$, $f_2=0.21$, $f_3=0.03$, 
$\eta=1/7$, $k = 0.03$, $\alpha_1=1/7$, and $\alpha_2=1/15$.

System \eqref{ex} is permanent with 
$\lambda^{L}= 1.876738171\times10^{-7}$, $\lambda^{U}=1.93$,
$M_1=65.83271997$, $M_2=8.722416333$, $M_3=0.017498412509$, $M_4=4.428264471$,
$M_6=0.02788991356$, $M=\max_{i=1,\ldots,6}(M_i)=M_1$,  
$m_1=58.16800031$, $m_2=7.7068897$, $m_3=0.0154611$,
$m_4=0.7093454$, $m_5=0.0000414$, $m_6=6.0482208\times10^{-7}$,
and $m= \min_{i=1,\ldots,6}(m_i)=m_6$. In addition, the conditions  
of Theorem~\ref{impth} are verified and we have
$$
4.653775371=A>B=4.148857053, \quad 
\psi=0.0505839, \quad 
1-\psi \mu(t)=0.94941603>0.
$$
We conclude that system \eqref{ex} has a unique positive 
almost periodic solution, which is uniformly asymptotic stable.
This is illustrated in Figure~\ref{fig}. 
\begin{figure}
\centering
\begin{subfigure}{0.4\textwidth}
\includegraphics[width=\textwidth]{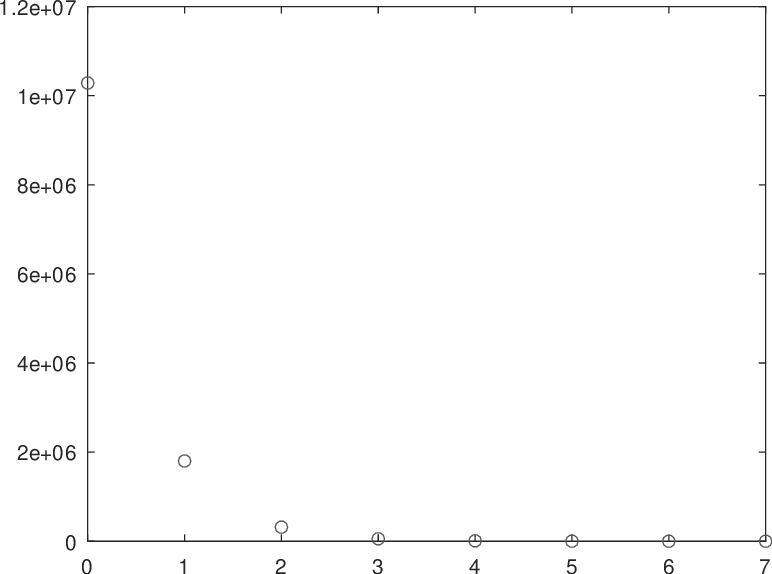}
\caption{Evolution of $x_1(t)$}
\label{fig:x1}
\end{subfigure}
\hfill
\begin{subfigure}{0.4\textwidth}
\includegraphics[width=\textwidth]{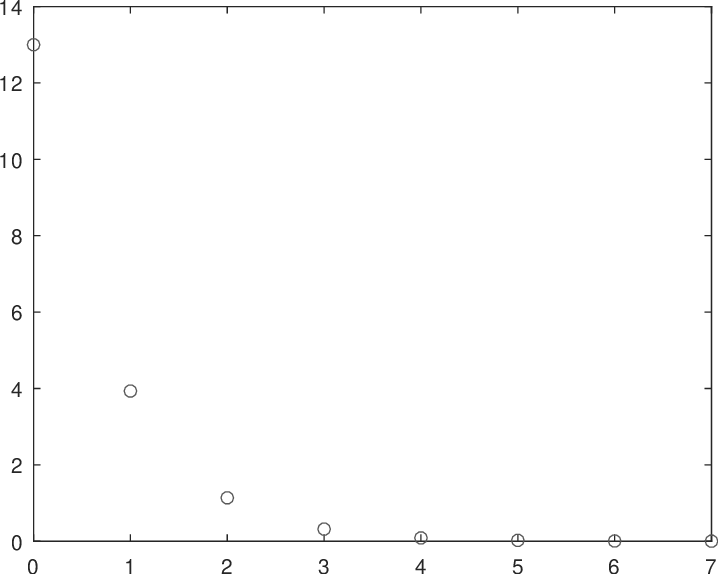}
\caption{Evolution of $x_2(t)$}
\label{fig:x2}
\end{subfigure}
\hfill
\begin{subfigure}{0.4\textwidth}
\includegraphics[width=\textwidth]{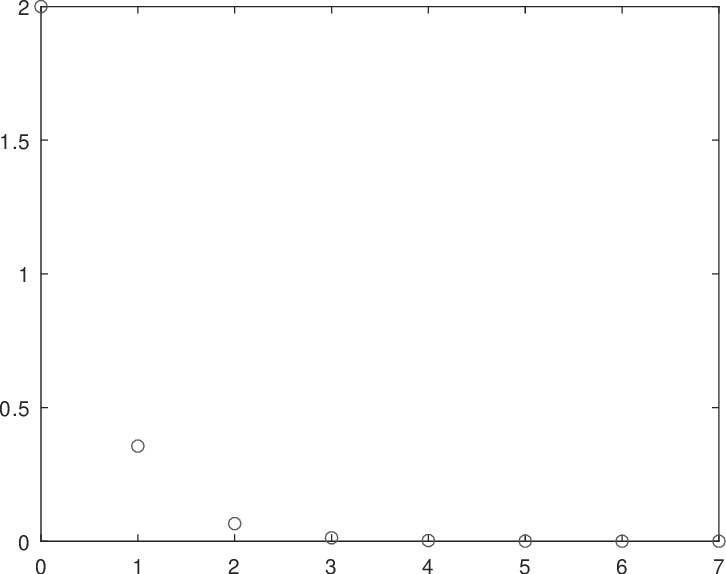}
\caption{Evolution of $x_3(t)$}
\label{fig:x3}
\end{subfigure}
\hfill
\begin{subfigure}{0.4\textwidth}
\includegraphics[width=\textwidth]{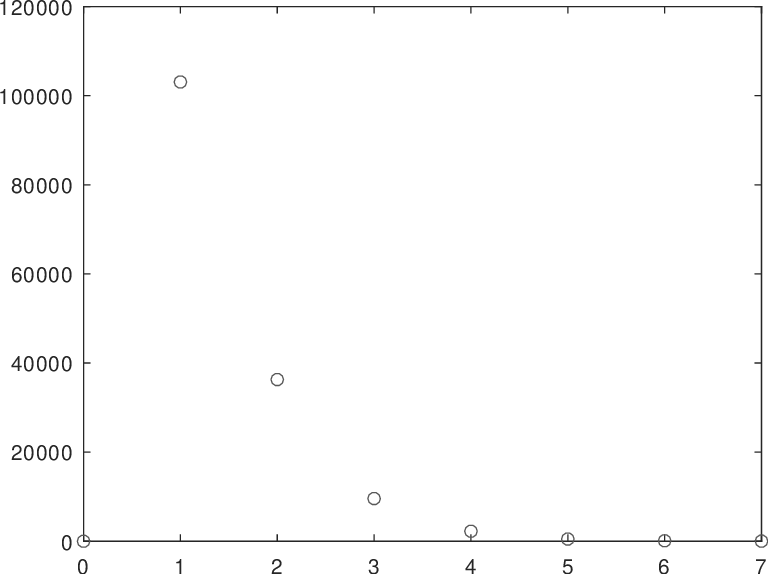}
\caption{Evolution of $x_4(t)$}
\label{fig:x4}
\end{subfigure}
\hfill
\begin{subfigure}{0.4\textwidth}
\includegraphics[width=\textwidth]{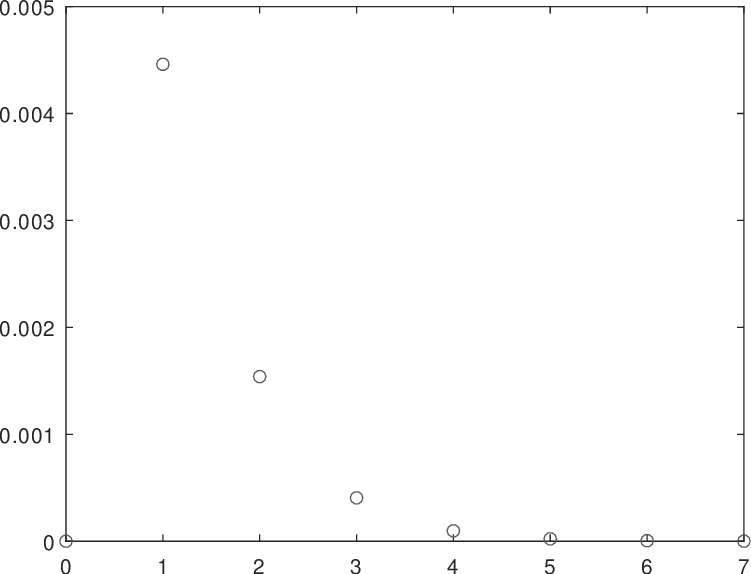}
\caption{Evolution of $x_5(t)$}
\label{fig:x5}
\end{subfigure}
\hfill
\begin{subfigure}{0.4\textwidth}
\includegraphics[width=\textwidth]{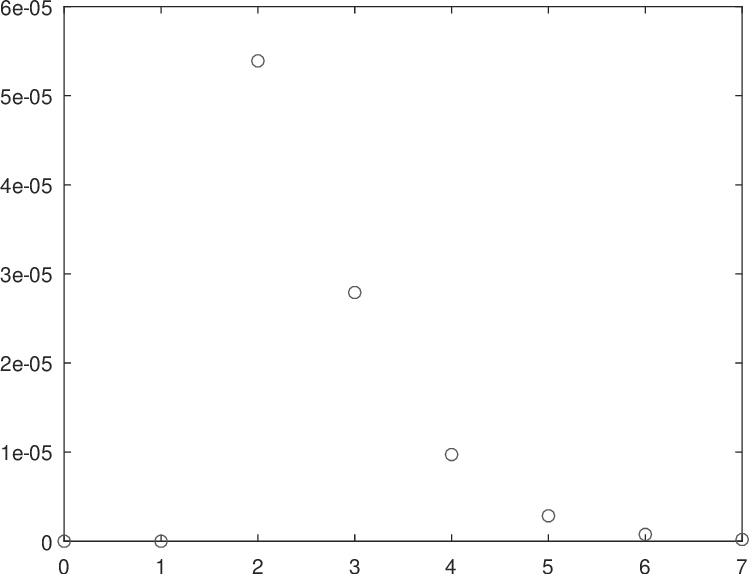}
\caption{Evolution of $x_6(t)$}
\label{fig:x6}
\end{subfigure}
\caption{Example~\ref{ex01}: solution of \eqref{ex} during 7 days.}
\label{fig}
\end{figure}
\end{ex}

\section*{Acknowledgments}

This work is part of first authors' PhD project.
Torres is funded by the Fundação para a Ci\^{e}ncia e a Tecnologia, I.P.
(FCT, Funder ID = 50110000187) under grants UIDB/04106/2020 
and UIDP/04106/2020 and the project 2022.03091.PTDC 
``Mathematical Modelling of Multi-scale Control Systems: 
applications to human diseases'' (CoSysM3), 
financially supported by national funds (OE) 
through FCT/MCTES.



\end{document}